\documentclass{svmult}
\usepackage[english]{babel}
\usepackage[T1]{fontenc}
\usepackage[utf8]{inputenc}

\usepackage{mathptmx}
\usepackage{helvet}
\usepackage{courier}
\usepackage{makeidx}
\usepackage{graphicx}
\usepackage{multicol}
\usepackage{amsfonts}
\usepackage{amssymb}
\usepackage{stmaryrd}
\usepackage{verbatim}
\usepackage{enumerate}
\usepackage{mathtools}

\smartqed

\mathtoolsset{showonlyrefs=true}
\allowdisplaybreaks[1]
\begin{document}
\bibliographystyle{spmpsci}

\title*{The law of large numbers for the free multiplicative convolution}
\author{Uffe Haagerup and Sören Möller}
\institute{Uffe Haagerup \at Department of Mathematical Sciences, University of Copenhagen, Universitetsparken 5, 2100 Copenhagen Ø, Denmark, \email{haagerup@math.ku.dk}
\and Sören Möller \at Department of Mathematics and Computer Science, University of Southern Denmark, Campusvej 55, 5230 Odense M, Denmark, \email{moeller@imada.sdu.dk}
\and {The first author is supported by ERC Advanced Grant No. OAFPG 27731 and the Danish National Research Foundation through the Center for Symmetry and Deformation.}}
\maketitle

\abstract{In classical probability the law of large numbers for the multiplicative convolution follows directly from the law for the additive convolution. In free probability this is not the case. The free additive law was proved by D. Voiculescu in 1986 for probability measures with bounded support and extended to all probability measures with first moment by J. M. Lindsay and V. Pata in 1997, while the free multiplicative law was proved only recently by G. Tucci in 2010. In this paper we extend Tucci's result to measures with unbounded support while at the same time giving a more elementary proof for the case of bounded support. In contrast to the classical multiplicative convolution case, the limit measure for the free multiplicative law of large numbers is not a Dirac measure, unless the original measure is a Dirac measure. We also show that the mean value of $\ln x$ is additive with respect to the free multiplicative convolution while the variance of $\ln x$ is not in general additive. Furthermore we study the two parameter family $(\mu_{\alpha,\beta})_{\alpha,\beta\ge0}$ of measures on $(0,\infty)$ for which the $S$-transform is given by $S_{\mu_{\alpha,\beta}}(z) = (-z)^\beta (1+z)^{-\alpha}$, $0 < z < 1$.}




\section{Introduction}
In classical probability the weak law of large numbers is well known (see for instance \cite[Corollary 5.4.11]{Rosenthal2006}), both for additive and multiplicative convolution of Borel measures on $\mathbb{R}$, respectively, $[0,\infty)$.

Going from classical probability to free probability, one could ask if similar results exist for the additive and multiplicative free convolutions $\boxplus$ and $\boxtimes$ as defined by D. Voiculescu in \cite{Voiculescu1986} and \cite{Voiculescu1987} and extended to unbounded probability measures by H. Bercovici and D. Voiculescu in \cite{BercoviciVoiculescu1993}. The law of large numbers for the free additive convolution of measures with bounded support is an immediate consequence of D. Voiculescu's work in \cite{Voiculescu1986} and J. M. Lindsay and V. Pata proved it for measures with first moment in 
\cite[Corollary 5.2]{LindsayPata1997}.
\index{law of large numbers, free additive}
\begin{theorem}[{\cite[Corollary 5.2]{LindsayPata1997}}]
Let $\mu$ be a probability measure on $\mathbb{R}$ with existing mean value $\alpha$, and let $\psi_n \colon \mathbb{R} \to \mathbb{R}$ be the map $\psi_n(x) = \frac{1}{n} x$. Then
\begin{align}
\dot{\psi}_n(\underbrace{\mu \boxplus \dots \boxplus \mu}_\text{n times}) & \to \delta_\alpha
\end{align}
where convergence is weak and $\delta_x$ denotes the Dirac measure at $x \in \mathbb{R}$.
\end{theorem}

Here $\dot{\phi}(\mu)$ denotes the image measure of $\mu$ under $\phi$ for a Borel measurable function $\phi \colon \mathbb{R} \to \mathbb{R}$, respectively, $[0,\infty) \to [0,\infty)$.

In classical probability the multiplicative law follows directly from the additive law. This is not the case in free probability, here a multiplicative law requires a separate proof. This has been proved by G. H. Tucci in \cite[Theorem 3.2]{Tucci2010} for measures with bounded support using results on operator algebras from \cite{HaagerupLarsen2000} and \cite{HaagerupSchultz2009}. In this paper we give an elementary proof of Tucci's theorem which also shows that the theorem holds for measures with unbounded support.
\index{law of large numbers, free multiplicative}
\begin{theorem}
\label{WLLN_main_theorem}
Let $\mu$ be a probability measure on $[0,\infty)$ and let $\phi_n \colon [0,\infty) \to [0,\infty)$ be the map $\phi_n(x) = x^{\frac{1}{n}}$. Set  $\delta = \mu(\{0\})$. If we denote
\begin{align}
\nu_n = \dot{\phi}_n(\mu_n) = \dot{\phi}_n(\underbrace{\mu \boxtimes \dots \boxtimes \mu}_\text{n times})
\end{align}
then $\nu_n$ converges weakly to a probability measure $\nu$ on $[0,\infty)$. If $\mu$ is a Dirac measure on $[0,\infty)$ then $\nu=\mu$. Otherwise $\nu$ is the unique measure on $[0,\infty)$ characterised by $\nu\left(\left[0,\frac{1}{S_\mu(t-1)}\right]\right) = t$ for all $t \in (\delta,1)$ and $\nu(\{0\})=\delta$. The support of the measure $\nu$ is the closure of the interval
\begin{align}
(a,b) & = \left( \left( \int_0^\infty x^{-1} \D\mu(x) \right)^{-1} , \int_0^\infty x \D\mu(x) \right),
\end{align}
where $0 \le a < b \le \infty$.
\end{theorem}

Note that unlike the additive case, the multiplicative limit distribution is only a Dirac measure if $\mu$ is a Dirac measure. Furthermore $S_\mu$ and hence (by \cite[Theorem 2.6]{Voiculescu1987}) $\mu$ can be reconstructed from the limit measure.

We start by recalling some definitions and proving some preliminary results in Section \ref{WLLN_section_preliminaries}, which then in Section \ref{WLLN_section_main_result} are used to prove Theorem \ref{WLLN_main_theorem}. In Section \ref{WLLN_section_further_formulas} we prove some further formulas in connection with the limit law, which we in Section \ref{WLLN_section_examples} apply to the two parameter family $(\mu_{\alpha,\beta})_{\alpha,\beta \ge 0}$ of measures on $(0,\infty)$ for which the $S$-transform is given by $S_{\mu_{\alpha,\beta}}(z) = \frac{(-z)^\beta}{(1+z)^\alpha}$, $0 < z < 1$.

\section{Preliminaries}
\label{WLLN_section_preliminaries}

We start with recalling some results we will use and proving some technical tools necessary for the proof of Theorem \ref{WLLN_main_theorem}. 
At first we recall the definition and some properties of Voiculescu's $S$-transform for measures on $[0,\infty)$ with unbounded support as defined by H. Bercovici and D. Voiculescu in \cite{BercoviciVoiculescu1993}.

\index{S-transform}
\begin{definition}[{\cite[Section 6]{BercoviciVoiculescu1993}}]
\label{WLLN_def_S_transform}
Let $\mu$ be a probability measure on $[0,\infty)$ and assume that $\delta = \mu(\{0\}) < 1$. We define $\psi_\mu(u) = \int_0^\infty \frac{tu}{1-tu} \D\mu(t)$ and denote its inverse in a neighbourhood of $(\delta-1,0)$ by $\chi_\mu$. Now we define the \emph{$S$-transform} of $\mu$ by $S_\mu(z) = \frac{z+1}{z} \chi_\mu(z)$ for $z \in (\delta-1,0)$.
\end{definition}

\begin{lemma}[{\cite[Proposition 6.8]{BercoviciVoiculescu1993}}]
\label{WLLN_lem_prop_Stransform}
Let $\mu$ be a probability measure on $[0,\infty)$ with $\delta = \mu(\{0\})<1$ then $S_\mu$ is decreasing on $(\delta-1,0)$ and positive. Moreover, if $\delta > 0$ we have $S_\mu(z) \to \infty$ if $z \to \delta - 1$.
\end{lemma}

\begin{lemma}
\label{WLLN_lemma_monotonicity_S}
Let $\mu$ be a probability measure on $[0,\infty)$ with $\delta = \mu(\{0\}) < 1$. Assume that $\mu$ is not a Dirac measure, then $S_\mu'(z) < 0$ for $z \in (\delta-1,0)$. In particular $S_\mu$ is strictly decreasing on $(\delta-1,0)$.
\end{lemma}

\begin{proof}
For $u \in (-\infty,0)$,
\begin{align}
\label{WLLN_eq_monotonicity_psi}
\psi_\mu'(u) & = \int_0^\infty \frac{t}{(1-ut)^2} \D\mu(t) > 0.
\end{align}
Moreover $\lim_{u \to 0-} \psi_\mu(u) = 0$ and $\lim_{u \to -\infty} \psi_\mu(u) = \delta - 1$.
Hence $\psi_\mu$ is a strictly increasing homeomorphism of $(-\infty,0)$ onto $(\delta-1,0)$. For $u \in (-\infty,0)$, we have
\begin{align}
S_\mu(\psi_\mu(u)) = \frac{\psi_\mu(u)+1}{\psi_\mu(u)} \cdot u.
\end{align}

Hence
\begin{align}
\label{WLLN_eq_logS_diff}
\frac{\D}{\D u}\left( \ln S_\mu(\psi_\mu(u))\right) = - \frac{\psi_\mu'(u)}{\psi_\mu(u)(\psi_\mu(u)+1)} + \frac{1}{u} = \frac{\psi_\mu(u)(\psi_\mu(u)+1) - u \psi_\mu'(u)}{u \psi_\mu(u)(\psi_\mu(u)+1)} \quad
\end{align}
where the denominator is positive and the nominator is equal to
\begin{align}
& \left(\int_0^\infty \frac{ut}{1-ut} \D\mu(t)\right) \cdot \left(\int_0^\infty \frac{1}{1-ut} \D\mu(t)\right) - \int_0^\infty \frac{ut}{(1-ut)^2} \D\mu(t) \\
& = \frac{u}{2} \int_0^\infty \int_0^\infty \frac{s+t}{(1-us)(1-ut)} \D\mu(s)\D\mu(t) \\
& \qquad - \frac{u}{2} \int_0^\infty \int_0^\infty \left(\frac{s}{(1-us)^2} + \frac{t}{(1-ut)^2}\right) \D\mu(s)\D\mu(t) \\
& = -\frac{u^2}{2} \int_0^\infty \int_0^\infty \frac{(s-t)^2}{(1-us)^2(1-ut)^2} \D\mu(s)\D\mu(t)
\end{align}
where we have used that
\begin{align}
(s+t)(1-us)(1-ut) - s(1-ut)^2 - t(1-us)^2 = -u(s-t)^2.
\end{align}
Since $\mu$ is not a Dirac measure,
\begin{align}
(\mu \times \mu)\left(\left\{(s,t) \in [0,\infty)^2 : s \neq t \right\} \right) > 0
\end{align}
and thus
\begin{align}
\int_0^\infty \int_0^\infty \frac{(s-t)^2}{(1-us)^2(1-ut)^2} \D\mu(s)\D\mu(t) > 0
\end{align}
which shows that the right hand side of \eqref{WLLN_eq_logS_diff} is strictly positive. Hence
\begin{align}
\frac{\D}{\D z}\left( \ln S_\mu(z)\right) < 0
\end{align} 
for $z \in (\delta - 1,0)$, which proves the lemma.
\qed
\end{proof}

\begin{remark}
Furthermore, by {\cite[Proposition 6.1]{BercoviciVoiculescu1993}} and {\cite[Proposition 6.3]{BercoviciVoiculescu1993}} $\psi_\mu$ and $\chi_\mu$ are analytic in a neighbourhood of $(-\infty,0)$, respectively, $(-1,0)$, hence $S_\mu$ is analytic in a neighbourhood of $(\delta-1,0)$.
\end{remark}

\begin{lemma}[{\cite[Corollary 6.6]{BercoviciVoiculescu1993}}]
\label{WLLN_lem_Smunu_eq_Smu_Snu}
Let $\mu$ and $\nu$ be probability measures on $[0,\infty)$, none of them beeing $\delta_0$, then we have $S_{\mu \boxtimes \nu} = S_\mu S_\nu$.
\end{lemma}

Next we have to determine the image of $S_\mu$. Here we closely follow the argument given for measures with compact support by F. Larsen and the first author in \cite[Theorem 4.4]{HaagerupLarsen2000}.

\begin{lemma}
\label{WLLN_S_image}
Let $\mu$ be a probability measure on $[0,\infty)$ not being a Dirac measure, then $S_\mu((\delta-1,0)) = (b^{-1},a^{-1})$,
where $a$, $b$ and $\delta$ are defined as in Theorem \ref{WLLN_main_theorem}.
\end{lemma}

\begin{proof}
First assume $\delta=0$. Observe that for $u \to \infty$ we have
\begin{align}
\int_0^\infty\frac{u}{1+ut} \D\mu(t)  \to \int_0^\infty\frac{1}{t} \D\mu(t) = a^{-1} \quad \text{ and } \quad
\int_0^\infty\frac{ut}{1+ut} \D\mu(t) & \to 1.
\end{align}
Hence
\begin{align}
\frac{-\psi_\mu(-u)}{u (\psi_\mu(-u)+1)} &= \left(\int_0^\infty\frac{ut}{1+ut} \D\mu(t)\right)\left(\int_0^\infty\frac{u}{1+ut} \D\mu(t)\right)^{-1} \to a \quad \text{ for } u \to \infty.
\end{align}
Similarly for $u \to 0$ we have
\begin{align}
\int_0^\infty\frac{t}{1+ut} \D\mu(t) & \to \int_0^\infty t \D\mu(t) = b \quad \text{ and } \quad
\int_0^\infty\frac{1}{1+ut} \D\mu(t) \to 1.
\end{align}
Hence
\begin{align}
\frac{-\psi_\mu(-u)}{u (\psi_\mu(-u)+1)} &= \frac{\int_0^\infty\frac{t}{1+ut} \D\mu(t)}{\int_0^\infty\frac{1}{1+ut} \D\mu(t)} \to b \quad \text{ for } u \to 0.
\end{align}

As $\chi_\mu$ is the inverse of $\psi_\mu$ we have
\begin{align}
\label{WLLN_eq_S_g}
S_\mu(\psi_\mu(-u)) &= \frac{\psi_\mu(-u)+1}{\psi_\mu(-u)} \chi_\mu(\psi_\mu(-u)) = \frac{u (\psi_\mu(-u)+1)}{-\psi_\mu(-u)}.
\end{align}

By \eqref{WLLN_eq_monotonicity_psi} and Lemma \ref{WLLN_lemma_monotonicity_S} $\psi_\mu$ is strictly increasing and continuous and $S_\mu$ is strictly decreasing and continuous so $S_\mu(\psi_\mu((-\infty,0))) = S_\mu((-1,0)) = (b^{-1},a^{-1})$.

If now $\delta>0$ we have by Lemma \ref{WLLN_lem_prop_Stransform} that $S_\mu(z) \to \infty$ for $z \to \delta-1$, so in this case continuity gives us $S_\mu((\delta-1,0)) = (b^{-1},\infty)$, which is as desired as $a=0$ in this case.
\qed
\end{proof}

\section{Proof of the main result}
\label{WLLN_section_main_result}
Let $\mu$ be a probability measure on $[0,\infty)$ and let $\nu$ be as defined in Theorem \ref{WLLN_main_theorem}. If $\mu$ is a Dirac measure, then $\nu_n=\mu$ for all $n$ and hence $\nu_n \to \nu=\mu$ weakly, so the theorem holds in this case. In the following we can therefore assume that $\mu$ is not a Dirac measure. We start by assuming further that $\mu(\{0\})=0$, and will deal with the case $\mu(\{0\})>0$ in Remark \ref{WLLN_remark_atom_zero}.

\begin{lemma}
\label{WLLN_lem_nun_equality}
For all $t \in (0,1)$ and all $n \ge 1$ we have
\begin{align}
\int_0^\infty \left(1+\frac{1-t}{t} S_{\mu}(t-1)^n x^n \right)^{-1} \D\nu_n(x) & = t.
\end{align}
\end{lemma}

\begin{proof}
Let $t \in (0,1)$ and set $z = t-1$. By Definition \ref{WLLN_def_S_transform} we have
\begin{align}
z + 1 & = \psi_{\mu_n}(\chi_{\mu_n}(z)) + 1 \\
& = \int_0^\infty \frac{\chi_{\mu_n}(z)x}{1-\chi_{\mu_n}(z)x} \D\mu_n(x) + 1 \\
& = \int_0^\infty \frac{1}{1-\chi_{\mu_n}(z)x} \D\mu_n(x) \\
& = \int_0^\infty \left(1-\frac{z}{z+1} S_{\mu_n}(z) x\right)^{-1} \D\mu_n(x) \\
& = \int_0^\infty \left(1-\frac{z}{z+1} S_{\mu}(z)^n x\right)^{-1} \D\mu_n(x).
\end{align}
In the last equality we use multiplicativity of the $S$-transform from Lemma \ref{WLLN_lem_Smunu_eq_Smu_Snu}.

Now substitute $t=z+1$ and afterwards $y^n = x$ and use the definition of $\nu_n$ to get
\begin{align}
t & = \int_0^\infty \left(1+\frac{1-t}{t} S_\mu(t-1)^n x \right)^{-1} \D\mu_n(x) \\
& = \int_0^\infty \left(1+\frac{1-t}{t} S_\mu(t-1)^n y^n \right)^{-1} \D\nu_n(y).
\end{align}
\qed
\end{proof}

Now, using this lemma, we can prove the following characterisation of the weak limit of $\nu_n$.

\begin{lemma}
\label{WLLN_lem_limit_nun}
For all $t \in (0,1)$ we have $t = \lim_{n\to\infty}\nu_n\left(\left[0,\frac{1}{S_\mu(t-1)}\right]\right)$.
\end{lemma}

\begin{proof}
Fix $t \in (0,1)$ and let $t' \in (0,t)$. Then
\begin{align}
\label{WLLN_eqn_inequality1}
t' &  = \int_0^\infty \left(1+\frac{1-t'}{t'} S_\mu(t'-1)^n x^n \right)^{-1} \D\nu_n(x) \\
& \le \int_0^\infty \left(1+\frac{1-t}{t} S_\mu(t'-1)^n x^n \right)^{-1} \D\nu_n(x) \\
& \le \int_0^\frac{1}{S_\mu(t-1)} 1 \D\nu_n(x) 
+ \int_\frac{1}{S_\mu(t-1)}^\infty \left(1+\frac{1-t}{t} S_\mu(t'-1)^n x^n \right)^{-1} \D\nu_n(x) \\
& \le \int_0^\frac{1}{S_\mu(t-1)} 1 \D\nu_n(x) 
+ \int_\frac{1}{S_\mu(t-1)}^\infty \left(1+\frac{1-t}{t} \left(\frac{S_\mu(t'-1)}{S_\mu(t-1)}\right)^n \right)^{-1} \D\nu_n(x) \\
& \le \nu_n\left(\left[0,\frac{1}{S_\mu(t-1)}\right]\right) + \left(1+\frac{1-t}{t} \left(\frac{S_\mu(t'-1)}{S_\mu(t-1)}\right)^n \right)^{-1}.
\end{align}
Here the first inequality holds as $t' \le t$ while $S_\mu(t'-1)^nx^n>0$, the second holds as $1+\frac{1-t}{t} S_\mu(t'-1)^n x^n \ge 0$, and the last because $\nu_n$ is a probability measure. 

By Lemma \ref{WLLN_lemma_monotonicity_S}, $S_\mu(t-1)$ is strictly decreasing, and hence $\frac{S_\mu(t'-1)}{S_\mu(t-1)} > 1$. This implies 
\begin{align}
\lim_{n \to \infty} \left(1+\frac{1-t}{t} \left(\frac{S_\mu(t'-1)}{S_\mu(t-1)}\right)^n \right)^{-1} = 0.
\end{align}
And hence
\begin{align}
t' & \le \liminf_{n \to \infty} \nu_n\left(\left[0,\frac{1}{S_\mu(t-1)}\right]\right).
\end{align}
As this holds for all $t' \in (0,t)$ we have
\begin{align}
\label{WLLN_eqn_mainInequality1}
t & \le \liminf_{n \to \infty} \nu_n\left(\left[0,\frac{1}{S_\mu(t-1)}\right]\right).
\end{align}

On the other hand if $t'' \in (t,1)$ we get
\begin{align}
t'' &  = \int_0^\infty \left(1+\frac{1-t''}{t''} S_\mu(t''-1)^nx^n \right)^{-1} \D\nu_n(x) \\
&  \ge \int_0^\infty \left(1+\frac{1-t}{t} S_\mu(t''-1)^nx^n \right)^{-1} \D\nu_n(x) \\
& \ge \int_0^\frac{1}{S(t-1)} \left(1+\frac{1-t}{t} S_\mu(t''-1)^n x^n \right)^{-1} \D\nu_n(x) \\
& \ge \int_0^\frac{1}{S(t-1)} \left(1+\frac{1-t}{t} \frac{S_\mu(t''-1)^n}{S_\mu(t-1)^n} \right)^{-1} \D\nu_n(x) \\
& \ge \nu_n\left(\left[0,\frac{1}{S_\mu(t-1)}\right]\right) \cdot \left(1+\frac{1-t}{t} \left(\frac{S_\mu(t''-1)}{S_\mu(t-1)}\right)^n \right)^{-1}.
\end{align}
Here the first inequality holds as $t''>t$ while $S_\mu(t''-1)x^n \ge 0$, and the second to last inequality holds as $S_\mu(t-1)$ is decreasing.

Again as $S_\mu(t-1)$ is strictly decreasing we have $\frac{S_\mu(t''-1)}{S_\mu(t-1)} < 1$, hence
\begin{align}
\lim_{n \to \infty} \left(1+\frac{1-t}{t} \left(\frac{S_\mu(t''-1)}{S_\mu(t-1)}\right)^n \right)^{-1} & = 1.
\end{align}

This implies
\begin{align}
t'' & \ge \limsup_{n \to \infty} \nu_n\left(\left[0,\frac{1}{S_\mu(t-1)}\right]\right).
\end{align}
As this holds for all $t'' \in (t,1)$ we have
\begin{align}
\label{WLLN_eqn_mainInequality2}
t & \ge \limsup_{n \to \infty} \nu_n\left(\left[0,\frac{1}{S_\mu(t-1)}\right]\right).
\end{align}

Combining \eqref{WLLN_eqn_mainInequality1} and \eqref{WLLN_eqn_mainInequality2} we get
\begin{align}
t & = \lim_{n \to \infty} \nu_n\left(\left[0,\frac{1}{S_\mu(t-1)}\right]\right)
\end{align}
as desired.
\qed
\end{proof}

For proving weak convergence of $\nu_n$ to $\nu$ it remains to show that $\nu_n$ vanishes in limit outside of the support of $\nu$.

\begin{lemma}
\label{WLLN_lem_vn_v_outside_support}
For all $x \le a$ and $y \ge b$ we have $\nu_n([0,x]) \to 0$, respectively, $\nu_n([0,y]) \to 1$.
\end{lemma}

\begin{proof}
To prove the first convergence, let $t \le a$ and $s \in (0,1)$. Now we have that $t \le \frac{1}{S_\mu(s-1)}$ from Lemma \ref{WLLN_S_image} and hence
\begin{align}
\limsup_{n \to \infty} \nu_n([0,t]) \le \limsup_{n \to \infty} \nu_n\left(\left[0,\frac{1}{S_\mu(s-1)}\right]\right) & = s. 
\end{align}
Here the inequality holds because $\nu_n$ is a positive measure and the equality comes from Lemma \ref{WLLN_lem_limit_nun}.
As this holds for all $s \in (0,1)$ we have $\limsup_{n \to \infty} \nu_n([0,t]) \le 0$ and hence $\limsup_{n \to \infty} \nu_n([0,t]) = 0$ by positivity of the measure.

For the second convergence we proceed in the same manner, by letting $t \ge b$ and $s \in (0,1)$. Now we have that $t \ge \frac{1}{S_\mu(s-1)}$ from Lemma \ref{WLLN_S_image} and hence
\begin{align}
\liminf_{n \to \infty} \nu_n([0,t]) \ge \liminf_{n \to \infty} \nu_n\left(\left[0,\frac{1}{S(s-1)}\right]\right) = s.
\end{align}
Again the inequality holds because $\nu_n$ is a positive measure and the equality comes from Lemma \ref{WLLN_lem_limit_nun}.
As this holds for all $s \in (0,1)$ we have $\limsup_{n \to \infty} \nu_n([0,t]) \ge 1$ and hence $\limsup_{n \to \infty} \nu_n([0,t]) = 1$ as $\nu_n$ is a probability measure.
\qed
\end{proof}

Lemmas \ref{WLLN_lem_limit_nun} and \ref{WLLN_lem_vn_v_outside_support} now prove Theorem \ref{WLLN_main_theorem} without any assumptions on bounded support as weak convergence of measures is equivalent to point-wise convergence of distribution functions for all but countably many $x \in [0,\infty)$.

\begin{remark}
\label{WLLN_remark_atom_zero}
In the case $\delta = \mu(\{0\}) > 0$, $S_\mu$ is only defined on $(\delta-1,0)$ and $S_\mu(z) \to \infty$ when $z \to \delta-1$.
This implies that Lemma \ref{WLLN_lem_nun_equality} only holds for $t \in (\delta,1)$, with a similar proof. Similarly Lemma \ref{WLLN_lem_limit_nun} only holds for $t \in (\delta,1)$, and in the proof we have to assume $t' \in (\delta,t)$. Similarly in the proof of Lemma \ref{WLLN_lem_vn_v_outside_support} we have to assume $s \in (\delta,1)$. Moreover, in Lemma \ref{WLLN_lem_vn_v_outside_support} the statement, $0 \le x \le a$ implies $\nu_n([0,x]) \to 0$ for $n \to \infty$, should be changed to $a=0$ and $\nu_n(\{0\}) = \delta = \nu(\{0\})$ for all $n \in \mathbb{N}$.
\end{remark}

Using our result we can prove the following corollary, generalizing a theorem (\cite[Theorem 2.2]{HaagerupSchultz2009}) by H. Schultz and the first author.

Let $(\mathcal{M},\tau)$ be a finite von Neumann algebra $\mathcal{M}$ with a normal faithful tracial state $\tau$. In \cite[Proposition 3.9]{HaagerupSchultz2007} the definition of Brown's spectral distribution measure $\mu_T$ was extended to all operators $T \in \mathcal{M}^\Delta$, where $\mathcal{M}^\Delta$ is the set of unbounded operators affiliated with $\mathcal{M}$ for which $\tau(\ln^+(|T|)) < \infty$.

\index{R-diagonal}
\begin{corollary}
If $T$ is an $R$-diagonal in $\mathcal{M}^\Delta$ then $\dot{\phi}(\mu_{(T^*)^n T^n}) \to \dot{\psi}(\mu_T)$ weakly, where $\psi(z) = |z|^2$, $z \in \mathbb{C}$, and $\phi_n(x)=x^{1/n}$ for $x \ge 0$
\end{corollary}

\begin{proof}
By \cite[Proposition 3.9]{HaagerupSchultz2007} we have $\mu_{T^*T}^{\boxtimes n} = \mu_{(T^*)^n T^n}$ and by Theorem \ref{WLLN_main_theorem} we have $\dot{\phi}(\mu_{T^*T}^{\boxtimes n}) \to \nu$ weakly. On the other hand observe that $\nu = \dot{\psi}(\mu_T)$ by \cite[Theorem 4.17]{HaagerupSchultz2007} which gives the result.
\qed
\end{proof}

\begin{remark}
In \cite[Theorem 1.5]{HaagerupSchultz2009} it was shown that $\dot{\phi}_n(\mu_{(T^*)^nT^n}) \to \dot{\psi}(\mu_T)$ weakly for all bounded operators $T \in \mathcal{M}$. It would be interesting to know, whether this limit law can be extended to all $T \in \mathcal{M}^\Delta$.
\end{remark}

\section{Further formulas for the $S$-transform}
\label{WLLN_section_further_formulas}

In this section we present some further formulas for the $S$-transform of measures on $[0,\infty)$, obtained by similar means as in the preceding sections and use those to investigate the difference between the laws of large numbers for classical and free probability. From now on we assume $\mu(\{0\}) = 0$. Therefore $\mu$ can be considered as a probability measure on $(0,\infty)$.

We start with a technical lemma which will be useful later.
\begin{lemma}
\label{WLLN_lem_ln_identities}
We have the following identities
\begin{align}
\int_0^1 \ln^2\left(\frac{t}{1-t}\right) \D t & = \frac{\pi^2}{3} \\
\int_0^1 \ln^2 t \D t & = 2 \\
\int_0^1 \ln^2(1-t) \D t & = 2 \\
\int_0^1 \ln t \ln(1-t) \D t & = 2 - \frac{\pi^2}{6}.
\end{align}
\end{lemma}

\begin{proof}
For the first identity we start with the substitution $x = \frac{t}{1-t}$ which gives us $t = \frac{x}{1+x}$ and $\D t = \frac{\D x}{(1+x)^2}$ and hence
\begin{align}
\int_0^1 \ln^2\left(\frac{t}{1-t}\right) \D t & = \int_0^\infty \frac{\ln^2 x}{(1+x)^2} \D x \\
& = \left. \frac{\D^2}{\D \alpha^2} \int_0^\infty \frac{x^\alpha}{(1+x)^2} \D x \right|_{\alpha=0} \\
& = \left. \frac{\D^2}{\D \alpha^2} B(1+\alpha,1-\alpha) \right|_{\alpha=0} \\
& = \left. \frac{\D^2}{\D \alpha^2} \frac{\pi \alpha}{\sin(\pi \alpha)} \right|_{\alpha=0} \\
& = \left. \frac{\D^2}{\D \alpha^2} \left( 1 - \frac{(\pi \alpha)^2}{3!} + \cdots \right)^{-1} \right|_{\alpha=0} \\
& = \left. \frac{\D^2}{\D \alpha^2} \left( 1 + \frac{\pi^2}{6} \alpha^2 + \cdots \right) \right|_{\alpha=0} = \frac{\pi^2}{3}
\end{align}
where $B(\cdot,\cdot)$ denotes the Beta function. The second and the third identity follow from the substitution $t \mapsto \exp(-x)$, respectively, $1-t \mapsto \exp(-x)$.

Finally, the last identity follows by observing
\begin{align}
\frac{\pi^2}{3} &= \int_0^1 \ln^2\left(\frac{t}{1-t}\right) \D t \\
& = \int_0^1 \ln^2 t + \ln^2(1-t) - 2 \ln t \ln(1-t) \D t \\
& = 4 - 2 \int_0^1 \ln t \ln(1-t) \D t
\end{align}
which gives the desired result.
\qed
\end{proof}

Now we prove two propositions calculating the expectations of $\ln x $ and $\ln^2 x$ both for $\mu$ and $\nu$  expressed by the $S$-transform of $\mu$.

\begin{proposition}
\label{WLLN_proposition_formula_ln}
Let $\mu$ be a probability measure on $(0,\infty)$ and let $\nu$ be as defined in Theorem \ref{WLLN_main_theorem}. Then $\int_0^\infty \left|\ln x \right| \D\mu(x) < \infty$ if and only if $\int_0^1 \left|\ln S_\mu(t-1) \right| \D t < \infty$ and if and only if $\int_0^\infty \left|\ln x \right| \D\nu(x) < \infty$. If these integrals are finite, then
\begin{align}
\int_0^\infty \ln x \D\mu(x) &= - \int_0^1 \ln S_\mu(t-1) \D t = \int_0^\infty \ln x \D\nu(x).
\end{align}
\end{proposition}

\begin{proof}
For $x>0$, put $\ln^+x = \max(\ln x, 0)$ and $\ln^-x=\max(-\ln x,0)$. Then one easily checks that
\begin{align}
\ln^+x \le \ln(x+1) \le \ln^+x + \ln 2
\end{align}
and by replacing $x$ by $\frac{1}{x}$ it follows that 
\begin{align}
\ln^- x \le \ln\left(\frac{x+1}{x}\right) \le \ln^-x + \ln 2.
\end{align}
Hence
\begin{align}
\int_0^\infty \ln^+x \D\mu(x) < \infty \Leftrightarrow \int_0^\infty \ln(x+1) \D\mu(x) < \infty
\end{align}
and
\begin{align}
\int_0^\infty \ln^-x \D\mu(x) < \infty \Leftrightarrow \int_0^\infty \ln\left(\frac{x+1}{x}\right) \D\mu(x) < \infty.
\end{align}
We prove next that
\begin{align}
\label{WLLN_eq_integral_ln_1}
\int_0^\infty \ln(x+1) \D\mu(x) = \int_0^\infty \ln^-u \psi_\mu'(-u) \D u
\end{align}
and
\begin{align}
\label{WLLN_eq_integral_ln_2}
\int_0^\infty \ln\left(\frac{x+1}{x}\right) \D\mu(x) = \int_0^\infty \ln^+u \psi_\mu'(-u) \D u.
\end{align}

Recall from \eqref{WLLN_eq_monotonicity_psi}, that 
\begin{align}
\psi_\mu'(-u) & = \int_0^\infty \frac{t}{(1+ut)^2} \D\mu(t), \quad u > 0.
\end{align}

Hence by Tonelli's Theorem
\begin{align}
\int_0^\infty \ln^+u \psi_\mu'(-u) \D u = \int_1^\infty \ln u \psi_\mu'(-u) \D u 
= \int_0^\infty \int_1^\infty \frac{x}{(1+ux)^2} \ln u \D u \D \mu(x)
\end{align}
and similarly
\begin{align}
\int_0^\infty \ln^-u \psi_\mu'(-u) \D u 
= \int_0^\infty \int_0^1 \frac{x}{(1+ux)^2} \ln\left(\frac{1}{u}\right) \D u \D \mu(x).
\end{align}

By partial integration, we have
\begin{align}
\int_1^\infty \frac{x}{(1+ux)^2} \ln u \D u = \left[-\frac{\ln u}{1+ux} + \ln\left(\frac{u}{1+ux}\right)\right]_{u=1}^{u=\infty} = \ln\left(\frac{x+1}{x}\right)
\end{align}
and similarly
\begin{align}
\int_0^1 \frac{x}{(1+ux)^2} \ln\left(\frac{1}{u}\right) \D u & = \left[\frac{\ln u}{1+ux} - \ln\left(\frac{u}{1+ux}\right)\right]_{u=0}^{u=1} \\
& = \left[\frac{ux}{1+ux}\ln u + \ln(1+ux)\right]_{u=0}^{u=1} = \ln(x+1)
\end{align}
which proves \eqref{WLLN_eq_integral_ln_1} and \eqref{WLLN_eq_integral_ln_2}. Therefore
\begin{align}
\int_0^\infty \left|\ln x \right| \D\mu(x) < \infty \Leftrightarrow \int_0^\infty \left|\ln u \right| \psi_\mu'(-u) \D u < \infty
\end{align}
and substituting $x=\psi_\mu(-u)+1$ we get
\begin{align}
\int_0^\infty \left|\ln u \right| \psi_\mu'(-u) \D u = \int_0^1 \left|\ln\left(-\chi_\mu(t-1)\right)\right| \D t = \int_0^1 \left|\ln\left(\frac{t}{1-t}\right) + \ln S_\mu(t-1)\right| \D t.
\end{align}
Since $\int_0^1 \left|\ln\left(\frac{t}{1-t}\right)\right| \D t < \infty$ it follows that
\begin{align}
\int_0^\infty \left|\ln u \right| \psi_\mu'(-u) \D u < \infty \Leftrightarrow \int_0^1 \left|\ln S_\mu(t-1)\right| \D t < \infty.
\end{align}

If $\mu$ is not a Dirac measure, the substitution $x=S_\mu(t-1)^{-1}, 0 < t < 1$ gives $t=\nu((0,x])$ for $ a < x < b$, where as before $a = \left( \int_0^\infty x^{-1} \D\mu(x) \right)^{-1}$ and $b = \int_0^\infty x \D\mu(x)$. The measure $\nu$ is concentrated on the interval $(a,b)$. Hence 
\begin{align}
\int_0^1 \left| \ln x \right| \D\nu(x) = \int_a^b \left| \ln x \right| \D\nu(x) = \int_0^1 \left| \ln\left(\frac{1}{S_\mu(t-1)}\right)\right| \D t = \int_0^1 \left| \ln S_\mu(t-1) \right| \D t.
\end{align}

This proves the first statement in Proposition \ref{WLLN_proposition_formula_ln}. If all three integrals in that statement are finite, we get
\begin{align}
\int_0^\infty \ln x \D\mu(x) & = \int_0^\infty \ln(x+1) \D\mu(x) - \int_0^\infty \ln\left(\frac{x+1}{x}\right) \D\mu(x) \\
& = \int_0^\infty \left(\ln^-u - \ln^+ u\right)\psi_\mu'(-u) \D u = -\int_0^\infty \ln u \psi_\mu'(-u) \D u.
\end{align}

By the substitution $t=\psi_\mu(-u)+1$ we get
\begin{align}
\int_0^1 \ln\left(-\chi_\mu(t-1)\right) \D t &= \int_0^1 \left(\ln\left(\frac{1-t}{t}\right) + \ln S_\mu(t-1)\right) \D t = \int_0^1 \ln S_\mu(t-1) \D t.
\end{align}

Hence $\int_0^\infty \ln x \D \mu(x) = - \int_0^1 \ln S_\mu(t-1) \D t$. Moreover, by the substitution $x = S_\mu(t-1)^{-1}, 0 < t < 1$ we get
\begin{align}
\int_0^\infty \ln x \D\mu(x) &= \int_0^1 \ln \left(\frac{1}{S_\mu(t-1)}\right) \D t = \int_0^\infty \ln x \D\nu(x).
\end{align}
Finally, if $\mu=\delta_x$, $x \in (0,\infty)$, this identity holds trivially, because $\nu=\delta_x$ and $S_\nu(z) = \frac{1}{x}, 0 < z < 1$.
\qed
\end{proof}

\begin{corollary}
Let $\mu_1$ and $\mu_2$ be probability measures on $(0,\infty)$. If $\mathbb{E}_{\mu_1}(\ln x)$ and $\mathbb{E}_{\mu_2}(\ln x)$ exist then $\mathbb{E}_{\mu_1 \boxtimes \mu_2}(\ln x)$ also exists and
\begin{align}
\mathbb{E}_{\mu_1 \boxtimes \mu_2}(\ln x) = \mathbb{E}_{\mu_1}(\ln x) + \mathbb{E}_{\mu_2}(\ln x)
\end{align}
where $\mathbb{E}_\mu(f) = \int_0^\infty f(x) \D \mu(x)$.
\end{corollary}

\begin{proof}
The statement follows directly from Proposition \ref{WLLN_proposition_formula_ln} and multiplicativity of the $S$-transform.
\qed
\end{proof}

For further use, we define the map $\rho$ for a probability measure $\mu$ on $(0,\infty)$ by
\begin{align}
\rho(\mu) &= \int_0^1 \ln\left(\frac{1-t}{t}\right) \ln S_\mu(t-1) \D t.
\end{align}

Note that $\rho(\mu)$ is well-defined and non-negative for all probability measures on $(0,\infty)$ because
\begin{align}
\label{WLLN_eq_rho_left_right}
\ln\left(\frac{1-t}{t}\right) \ln S_\mu(t-1) &= \ln\left(\frac{1-t}{t}\right)\ln\left(\frac{S_\mu(t-1)}{S_\mu(-\frac{1}{2})}\right) + \ln\left(\frac{1-t}{t}\right)S_\mu\left(-\frac{1}{2}\right),
\end{align}
where the first term on the right hand side is non-negative for all $t \in (0,1)$ and the second term is integrable with integral $0$.

\begin{lemma}
\label{WLLN_lem_rho}
Let $\mu$ be a probability measure on $(0,\infty)$, then
\begin{align}
0 \le \rho(\mu) \le \frac{\pi}{\sqrt{3}} \left( \int_0^1 \ln^2 S_\mu(t-1) \D t \right)^{1/2}.
\end{align}
Furthermore, $\rho(\mu) = 0$ if and only if $\mu$ is a Dirac measure. Moreover, equality holds in the right inequality if and only if $S_\mu(z) = \left(\frac{z}{1+z}\right)^\gamma$ for some $\gamma>0$ and in this case $\rho(\mu) = \gamma \frac{\pi^2}{3}$.
Additionally, if $\mu_1, \mu_2$ are probability measures on $(0, \infty)$ we have $\rho(\mu_1 \boxtimes \mu_2) = \rho(\mu_1) + \rho(\mu_2)$.
\end{lemma}

\begin{proof}
We already have observed $\rho \ge 0$. For the second inequality observe that
\begin{align}
\rho(\mu)^2 \le \left( \int_0^1 \ln^2\left(\frac{1-t}{t}\right) \D t \right)\left(\int_0^1 \ln^2 S_\mu(t-1) \D t \right)
\end{align}
by the Cauchy-Schwarz-inequality, where the first term equals $\frac{\pi^2}{3}$ by Lemma \ref{WLLN_lem_ln_identities}.

If $\mu=\delta_a$ for some $a>0$ we have $S_\mu(z) = \frac{1}{a}$, hence $\ln S_\mu(t-1)$ is constant so the oddity of $\ln(\frac{1-t}{t})$ gives us $\rho(\mu)=0$. On the other hand, if $\rho(\mu)=0$,  the first term in \eqref{WLLN_eq_rho_left_right} has to integrate to $0$, but by symmetry of $\ln\left(\frac{1-t}{t}\right)$ and the fact that $S_\mu$ is decreasing, this implies that $S_\mu$ must be constant, hence $\mu$ is a Dirac measure.

Equality in the second inequality, by the Cauchy-Schwarz inequality happens precisely if $\ln S_\mu(t-1) = \gamma \ln(\frac{1-t}{t})$ for some $\gamma > 0$ which is the case if and only if $S_\mu(t-1) = \left(\frac{1-t}{t}\right)^\gamma$, and in this case $\rho(\mu) = \gamma \frac{\pi^2}{3}$ by Lemma \ref{WLLN_lem_ln_identities}.

For the last formula we use multiplicity of the $S$-transform to get
\begin{align}
\rho(\mu_1 \boxtimes \mu_2) &= \int_0^1 \ln\left(\frac{1-t}{t}\right) \ln S_{\mu_1 \boxtimes \mu_2}(t-1) \D t \\
&= \int_0^1 \ln\left(\frac{1-t}{t}\right)\left( \ln S_{\mu_1}(t-1) + \ln S_{\mu_2}(t-1) \right) \D t \\
&= \rho(\mu_1) + \rho(\mu_2).
\end{align}
\qed
\end{proof}

\begin{proposition}
\label{WLLN_proposition_formula_ln2}
Let $\mu$ be a probability measure on $(0,\infty)$, and let $\nu$ be defined as in Theorem \ref{WLLN_main_theorem}. Then
\begin{align}
\int_0^\infty \ln^2 x \D\mu(x) & = \int_0^1 \ln^2 S_\mu(t-1) \D t +  2 \rho(\mu) \\
\int_0^\infty \ln^2 x \D\nu(x) & = \int_0^1 \ln^2 S_\mu(t-1) \D t \\
\mathbb{V}_\mu(\ln x) &= \mathbb{V}_\nu(\ln x) + 2\rho(\mu).
\end{align}
as equalities of numbers in $[0,\infty]$, where $\mathbb{V}_\sigma(\ln x)$ denotes the variance of $\ln x$ with respect to a probability measure $\sigma$ on $(0,\infty)$. 
Moreover 
\begin{align}
0 \le \rho(\mu) \le \frac{\pi}{\sqrt{3}} \mathbb{V}_\nu(\ln x)^{\frac{1}{2}}.
\end{align}
\end{proposition}

\begin{proof}
We first prove the following identity
\begin{align}
\label{WLLN_eq_formula_ln2_1}
\int_0^\infty \ln^2 u \psi_\mu'(-u) \D u &= \int_0^\infty \ln^2 x \D \mu(x) + \frac{\pi^2}{3}.
\end{align}
Since $\psi'(-u) = \int_0^\infty \frac{x}{(1+ux)^2} \D x$, we get by Tonelli's Theorem, that
\begin{align}
\int_0^\infty \ln^2 u \psi_\mu'(-u) \D u &= \int_0^\infty \left( \int_0^\infty \ln^2 u \frac{x}{(1+ux)^2} \D u \right) \D \mu(x) \\
& = \int_0^\infty \left( \int_0^\infty \ln^2 \left(\frac{v}{x}\right) \frac{\D v}{(1+v)^2} \right) \D \mu(x).
\end{align}
Note next that
\begin{align}
\int_0^\infty \ln^2 \left(\frac{v}{x}\right) \frac{\D v}{(1+v)^2} = c_0 + c_1 \ln x + c_2 \ln^2 x
\end{align}
where $c_0 = \int_0^\infty \frac{\ln^2 v}{(1+v)^2} \D v$, $c_1 = -2\int_0^\infty \frac{\ln v}{(1+v)^2} \D v$, and $c_2 = \int_0^\infty \frac{1}{(1+v)^2} \D v = 1$. Moreover, by the substitution $v = \frac{1}{w}$ one gets $c_1 = -c_1$ and hence $c_1=0$. Finally, by the substitution $v = \frac{t}{1-t}, 0 < t < 1$ and Lemma \ref{WLLN_lem_ln_identities},
\begin{align}
c_0 = \int_0^1 \ln^2\left(\frac{t}{1-t}\right) \D t = \frac{\pi^2}{3}.
\end{align}
Hence
\begin{align}
\int_0^\infty \ln^2 u \psi_\mu(-u) \D u = \int_0^\infty \left( \ln^2 x + \frac{\pi^2}{3}\right) \D\mu(x)
\end{align}
which proves \eqref{WLLN_eq_formula_ln2_1}. Next by the substitution $t=\psi_\mu(-u)+1$, we have
\begin{align}
\label{WLLN_eq_formula_ln2_2}
\int_0^\infty \ln^2 u \psi_\mu'(-u)\D u = \int_0^1 \ln^2\left(-\chi_\mu(t-1)\right)\D t = \int_0^1 \left( \ln\left(\frac{1-t}{t}\right)+\ln S_\mu(t-1)\right)^2\D t.
\end{align}
Since $t \mapsto \ln\left(\frac{1-t}{t}\right)$ is square integrable on $(0,1)$ the right hand side of \eqref{WLLN_eq_formula_ln2_2} is finite if and only if 
\begin{align}
\int_0^1 \ln\left(S_\mu(t-1)\right)^2 \D t < \infty.
\end{align}
Hence by \eqref{WLLN_eq_formula_ln2_1} and \eqref{WLLN_eq_formula_ln2_2} this condition is equivalent to
\begin{align}
\int_0^\infty \ln^2 x \D\mu(x) < \infty,
\end{align}
so to prove the first equation in Proposition \ref{WLLN_proposition_formula_ln2} is suffices to consider the case, where the two above integrals are finite. In that case $\rho(\mu) < \infty$ by Lemma \ref{WLLN_lem_rho}. Thus by Lemma \ref{WLLN_lem_ln_identities} and the definition of $\rho(\mu)$,
\begin{align}
\int_0^1 \left(\ln\left(\frac{1-t}{t}\right) + \ln S_\mu(t-1)\right)^2 \D t = \int_0^1 \ln^2\left(S_\mu(t-1)\right)\D t + 2 \rho(\mu) + \frac{\pi^2}{3}.
\end{align}
Hence by \eqref{WLLN_eq_formula_ln2_1} and \eqref{WLLN_eq_formula_ln2_2}
\begin{align}
\int_0^\infty \ln^2 x \D\mu(x) = \int_0^1 \ln^2\left(S_\mu(t-1)\right)\D t + 2 \rho(\mu).
\end{align}

The second equality in Proposition \ref{WLLN_proposition_formula_ln2}
\begin{align}
\int_0^\infty \ln^2 x \D\nu(x) & = \int_0^1 \ln^2 S_\mu(t-1) \D t
\end{align}
follows from the substitution $x=S_\mu(t-1)^{-1}$ in case $\mu$ is not a Dirac measure, and it is trivially true for Dirac measures. By the first two equalities in Proposition \ref{WLLN_proposition_formula_ln2}, we have
\begin{align}
\label{WLLN_eq_formula_ln2_3}
\int_0^\infty \ln^2 x \D\mu(x) = \int_0^\infty \ln^2 x \D\nu(x) + 2 \rho(\mu).
\end{align}
If both sides of this equality are finite, then by Proposition \ref{WLLN_proposition_formula_ln},
\begin{align}
\int_0^\infty \ln x \D\mu(x) = \int_0^\infty \ln x \D\nu(x)
\end{align}
where both integrals are well-defined. Combined with \eqref{WLLN_eq_formula_ln2_3} we get
\begin{align}
\label{WLLN_eq_formula_ln2_4}
\mathbb{V}_\mu(\ln x) = \mathbb{V}_\nu(\ln x) + 2\rho(\mu)
\end{align}
and if $\int_0^\infty \ln^2 x \D\mu(x) = + \infty$, both sides of \eqref{WLLN_eq_formula_ln2_4} must be infinite by \eqref{WLLN_eq_formula_ln2_3}.

As the $S$-transform behaves linearly when scaling the probability distribution in the sense that the image measure $\mu_c$ of $\mu$ under $x \mapsto cx$ for $c>0$ gives us $S_{\mu_c}(z) = c^{-1} S_{\mu}(z)$ we have for $\rho$ that 
\begin{align}
\rho(\mu_c) &= \int_0^1 \ln\left(\frac{1-t}{t}\right) \ln(c^{-1} S_{\mu}(t-1)) \D t \\
& = \int_0^1 \ln\left(\frac{1-t}{t}\right) \ln S_{\mu}(t-1) \D t + \int_0^1 \ln\left(\frac{1-t}{t}\right) c^{-1} \D t = \rho(\mu) + 0
\end{align}
by anti-symmetry of the second term around $t=\frac{1}{2}$. Using this for $c = \exp\left(\mathbb{E}_\nu(\ln x)\right)$, we get
\begin{align}
\rho(\mu)=\rho(\mu_c) & \le \frac{\pi}{\sqrt{3}}\left(\int_0^1 \left( \ln S_\mu(t-1) - \mathbb{E}_\nu\left(\ln x\right)\right)^2 \D t\right)^{\frac{1}{2}} \\
& = \frac{\pi}{\sqrt{3}}\left(\int_0^1 \left( \ln S_\mu(t-1)^2 - 2\mathbb{E}_\nu\left(\ln x\right)^2 + \mathbb{E}_\nu\left(\ln x\right)^2\right) \D t\right)^{\frac{1}{2}} \\
& = \frac{\pi}{\sqrt{3}} \left(\mathbb{V}_\nu(\ln x)\right)^{\frac{1}{2}}.
\end{align}
\qed
\end{proof}

Now we can use the preceeding lemmas to investigate the different behavior of the multiplicative law of large numbers in classical and free probability.
Note that in classical probability for a family of identically distributed independent random variables $(X_i)_{i=1}^\infty$ we have the identity $\mathbb{V}( \ln(\prod_{i=0}^n X_i) ) = n \mathbb{V}( \ln X_1 )$.
In free probability by Propositions \ref{WLLN_proposition_formula_ln} and \ref{WLLN_proposition_formula_ln2} we have instead
\begin{align}
& \mathbb{V}_{\mu^{\boxtimes n}}( \ln t ) \\
& = \int_0^\infty \ln^2 t \D(\mu^{\boxtimes n})(t) - \left( \int_0^\infty \ln t \D(\mu^{\boxtimes n})(t) \right)^2 \\
& = \int_0^1 \ln^2 S_{\mu^{\boxtimes n}}(t-1) \D z +  2 \rho(\mu^{\boxtimes n}) - \left( - \int_{-1}^0 \ln S_{\mu^{\boxtimes n}}(z) \D z \right)^2 \\
& = n^2 \int_0^1 \ln^2 S_{\mu}(t-1) \D z + 2 n \rho(\mu) - n^2 \left( \int_{-1}^0 \ln S_{\mu}(z) \D z \right)^2 \\
& = n^2 \mathbb{V}_\nu( \ln x ) + 2n \rho(\mu).
\end{align}
Hence $\mathbb{V}_{\mu^{\boxtimes n}}(\ln t) = n \mathbb{V}_\mu(\ln t) + n(n-1) \mathbb{V}_\nu(\ln t) > n \mathbb{V}_\mu(\ln t)$ for $n \ge 2$ if $\mu$ is not a Dirac measure and $\mathbb{V}_\nu(\ln t) < \infty$, which shows that the variance of $\ln t$ is not in general additive.

\begin{lemma}
\label{WLLN_lemma_formula_gamma}
Let $\mu$ be a probability measure on $(0,\infty)$ and let $\nu$ be defined as in Theorem \ref{WLLN_main_theorem}. Then
\begin{align}
\int_0^\infty x^\gamma \D\mu(x) = \frac{\sin(\pi\gamma)}{\pi\gamma} \int_0^1 \left( \frac{1-t}{t} S_\mu(t-1) \right)^{-\gamma} \D t
\end{align}
for $-1 < \gamma < 1$ and
\begin{align}
\int_0^\infty x^\gamma \D\nu(x) &= \int_0^1 S_\mu(t-1)^{-\gamma} \D t
\end{align}
for $\gamma \in \mathbb{R}$ as equalities of numbers in $[0,\infty]$.
\end{lemma}

\begin{proof}
By Tonelli's theorem followed by the substitution $u = y x$ we get
\begin{align}
\int_0^\infty y^{-\gamma} \psi_\mu'(-y) \D y & = \int_0^\infty \int_0^\infty \frac{y^{-\gamma}x}{(1+yx)^2} \D t \D\mu(x) \\
& = \int_0^\infty x^\gamma \int_0^\infty \frac{u^{-\gamma}}{(1+u)^2} \D u \D\mu(x) \\
& = B(1-\gamma,1+\gamma) \int_0^\infty x^\gamma \D\mu(x),
\end{align}
where  $B(s,t) = \int_0^\infty \frac{u^{s-1}}{(1+u)^{s+t}} \D u$ is the Beta function. But $B(1-\gamma,1+\gamma) =  \frac{\sin(\pi\gamma)}{\pi\gamma}$ by well-known properties of $B$. Substitute now $x = -\chi_\mu(-z)$ and $z=1-t$ to get
\begin{align}
\int_0^\infty x^{-\gamma} \psi_\mu'(-x) \D x  & = \int_0^1 \left(- \chi_\mu(-z)\right)^{-\gamma} \D z = \int_{0}^1 \left( \frac{1-t}{t} S_\mu(t-1)\right)^{-\gamma} \D t,
\end{align}
which gives the first identity.
The second identity follows from the substitution $x = S_\mu(t-1)^{-1}$ and the properties of $\nu$ from Theorem \ref{WLLN_main_theorem}.
\qed
\end{proof}

\section{Examples}
\label{WLLN_section_examples}

In this section we will investigate a two parameter family of distributions for which there can be made explicit calculations.

\begin{proposition}
\label{WLLN_prop_mu_alpha_beta}
Let $\alpha, \beta \ge 0$. There exists a probability measure $\mu_{\alpha,\beta}$ on $(0,\infty)$ which $S$-transform is given by
\begin{align}
S_{\mu_{\alpha,\beta}}(z) = \frac{(-z)^\beta}{(1+z)^\alpha}.
\end{align}
Furthermore, these measures form a two-parameter semigroup, multiplicative under $\boxtimes$ induced by multiplication of $(\alpha,\beta) \in [0,\infty) \times [0,\infty)$.
\end{proposition}

\begin{proof}
Note first that $\alpha=\beta=0$ gives $S_{\mu_{0,0}}=1$, which by uniqueness of the $S$-transform results in $\mu_{0,0}=\delta_1$, hence we can in the following assume $(\alpha,\beta) \neq (0,0)$.

Define the function $v_{\alpha,\beta} \colon \mathbb{C} \setminus [0,1] \to \mathbb{C}$ by
\begin{align}
v_{\alpha,\beta}(z) = \beta \ln(-z) - \alpha \ln(1+z)
\end{align}
for all $z \in \mathbb{C} \setminus [0,1]$.

In the following we for $z \in \mathbb{C}$ denote by $\arg{z} \in [-\pi,\pi]$ its argument.
Assume $z=x+\I y$ and $y > 0$ then
\begin{align}
\ln(-z) = \frac{1}{2}\ln\left(x^2+y^2\right) + \I \arg(-x-\I y)
\end{align}
where $\arg(-x-\I y) < 0 $, which implies that $\ln(\mathbb{C}^+) \subseteq \mathbb{C}^-$.
Similarly, if we assume $z=x+\I y$ and $y > 0$ then
\begin{align}
\ln(1+z) = \frac{1}{2}\ln\left((x+1)^2+y^2\right) + \I \arg((x+1)+\I y)
\end{align}
where $\arg((x+1)+\I y) > 0 $, which implies that $-\ln(1+\mathbb{C}^+) \subseteq \mathbb{C}^-$ and hence $v_{\alpha,\beta}(\mathbb{C}^+) \subseteq \mathbb{C}^-$.
Furthermore, we observe that for all $z \in \mathbb{C}$, $v_{\alpha,\beta}(\bar{z}) = \overline{v_{\alpha,\beta}(z)}$.
By \cite[Theorem 6.13 (ii)]{BercoviciVoiculescu1993} these results imply that there exists a unique $\boxtimes$-infinitely divisible measure $\mu_{\alpha,\beta}$ with the $S$-transform
\begin{align}
S_{\mu_{\alpha,\beta}}(z) = \exp(v(z)) = \exp(\beta \ln(-z) - \alpha \ln(1+z)) = \frac{(-z)^\beta}{(1+z)^\alpha}.
\end{align}

The semigroup property follows from multiplicativity of the $S$-transform.
\qed
\end{proof}

\index{free Bessel law}
\index{Boolean stable law}
The existence of $\mu_{\alpha,0}$ was previously proven by T. Banica, S. T. Belinschi, M. Capitaine and B. Collins in \cite{BanicaBelinschiCapitaineCollins2011} as a special case of free Bessel laws. The case $\mu_{\alpha,\alpha}$ is known as a Boolean stable law from O. Arizmendi and T. Hasebe \cite{ArizmendiHasebe2012p}.

Furthermore, there is a clear relationship between the measures $\mu_{\alpha,\beta}$ and $\mu_{\beta,\alpha}$.
\begin{lemma}
\label{WLLN_lem_ab_ba}
Let $\alpha, \beta \ge 0$, $(\alpha,\beta)\neq(0,0)$ and let $\zeta \colon (0,\infty) \to (0,\infty)$ be the map $\zeta(t)=t^{-1}$. Then we have
$\mu_{\beta,\alpha} = \dot{\zeta}(\mu_{\alpha,\beta})$,
where $\dot{\zeta}$ denotes the image measure under the map $\zeta$.
\end{lemma}

\begin{proof}
Put $\sigma = \dot{\zeta}(\mu_{\alpha,\beta})$. Then by the proof of \cite[Proposition 3.13]{HaagerupSchultz2007},
\begin{align}
S_{\sigma}(z) = \frac{1}{S_{\mu_{\alpha,\beta}}(-1-z)} = \frac{(-z)^\alpha}{(1+z)^\beta} = S_{\mu_{\beta,\alpha}}
\end{align}
for $0<z<1$. Hence $\sigma = \mu_{\beta,\alpha}$.
\qed
\end{proof}

\begin{lemma}
Let $(\alpha,\beta) \neq (0,0)$. Denote the limit measure corresponding to $\mu_{\alpha,\beta}$ by $\nu_{\alpha,\beta}$. Then $\nu_{\alpha,\beta}$ is uniquely determined by the formula
\begin{align}
F_{\alpha,\beta}\left(\frac{t^\alpha}{(1-t)^\beta}\right) &= t
\end{align}
for $0<t<1$, where $F_{\alpha,\beta}(x) = \nu_{\alpha,\beta}((0,x])$ is the distribution function of $\nu_{\alpha,\beta}$.
\end{lemma}

\begin{proof}
The lemma follows directly from Lemma \ref{WLLN_prop_mu_alpha_beta} and Theorem \ref{WLLN_main_theorem}.
\qed
\end{proof}

For $\beta=0$ and $\alpha>0$,
\begin{align}
F_{\alpha,0}(x) = \left\{ \begin{array}{ll} x^{\frac{1}{\alpha}}, & \qquad 0 < x < 1 \\ 1, & \qquad x \ge 1. \end{array} \right.
\end{align}
Similarly for $\alpha=0$ and $\beta>0$
\begin{align}
F_{0,\beta}(x) = \left\{ \begin{array}{ll} 0, & \qquad 0 < x < 1 \\ (1-x)^{-\frac{1}{\beta}}, & \qquad x \ge 1. \end{array} \right.
\end{align}
\index{Pareto distribution}
Hence $\nu_{0,\beta}$ is the Pareto distribution with scale parameter $1$ and shape parameter $\frac{1}{\beta}$.

\index{Burr distribution}
Moreover, if $\alpha=\beta>0$ we get
$F_{\alpha,\alpha}(x) = (1+x^{-1/\alpha})^{-1}$
for $x \in (0,\infty)$, which we recognize as the image measure of the Burr distribution with parameters $(1,\alpha^{-1})$ (or equivalently the Fisk or log-logistic distribution (cf. \cite[p. 54]{JohnsonKotzBalakrishnan1994}) with scale parameter $1$ and shape parameter $\alpha^{-1}$) under the map $x \mapsto x^{-1}$.

\index{free Poisson distribution}
On the other hand, we can make some observations about the distribution $\mu_{\alpha,\beta}$, too. For the cases $(\alpha,\beta)=(1,0)$ and $(\alpha,\beta)=(0,1)$ we can regognize the measures $\mu_{1,0}$ and $\mu_{0,1}$ from there $S$-transform, as
$S_{\mu_{1,0}}(z) = (1+z)^{-1}$
is the $S$-transform of the free Poisson distributions with shape parameter $1$ (cf. \cite[p. 34]{VoiculescuDykemaNica1992}), which is given by
\begin{align}
\mu_{1,0} = \frac{1}{2\pi} \sqrt{\frac{4-x}{x}} 1_{(0,4)}(x) \D x,
\end{align}
while $S_{\mu_{0,1}}(z) = -z$ according to Lemma \ref{WLLN_lem_ab_ba} is the $S$-transform of the image of the above free Poisson distribution under the map $t \mapsto t^{-1}$,
\begin{align}
\mu_{0,1} = \frac{1}{2\pi} \frac{\sqrt{4x-1}}{x^2} 1_{(\frac{1}{4},\infty)}(x) \D x,
\end{align}
\index{free stable distribution}
which is the same as the free stable distribution with parameters $\alpha=1/2$ and $\rho=1$ as described by H. Bercovici, V. Pata and P. Biane in \cite[Appendix A1]{BercoviciPataBiane1999}. More genereally, $\mu_{0,\beta}$ is the same as the free stable distribution $v_{\alpha,\rho}$ with $\alpha = \frac{1}{\beta+1}$ and $\rho=1$, because by \cite[Appendix A4]{BercoviciPataBiane1999} $v_{\alpha,1}$ is characterized by $\Sigma_{v_{\alpha,1}}(y) = \left(\frac{-y}{1-y}\right)^{\frac{1}{\alpha}-1}$, $y \in (-\infty,0)$, and it is easy to check that
\begin{align}
S_{v_{\alpha,0}}(z) =\Sigma_{v_{\alpha,0}}\left(\frac{z}{1+z}\right) = (-z)^{\frac{1}{\alpha}-1} = S_{\mu_{0,\frac{1}{\alpha}-1}}(z),\quad 0 < z < 1, 0 < \alpha < 1.
\end{align}

From the above observations, we now can describe a construction of the measures $\mu_{m,n}$.
\begin{proposition}
Let $m, n$ be nonnegative integers. Then the measure $\mu_{m,n}$ is given by
\begin{align}
\mu_{m,n} = \mu_{1,0}^{\boxtimes m} \boxtimes \mu_{0,1}^{\boxtimes n}.
\end{align}
\end{proposition}

\begin{proof}
By multiplicativity of the $S$-transform we have that
\begin{align}
S_{\mu_{1,0}^{\boxtimes m} \boxtimes \mu_{0,1}^{\boxtimes n}}(z) = S_{\mu_{1,0}}(z)^m S_{\mu_{0,1}}(z)^n = \frac{(-z)^n}{(1+z)^m} = S_{\mu_{m,n}}(z),
\end{align}
which by uniqueness of the $S$-transform gives the desired result.
\qed
\end{proof}

\begin{proposition}
For all $\alpha,\beta \ge 0$.
\begin{align}
\mathbb{E}_{\mu_{\alpha,\beta}}(\ln x) & = \beta - \alpha \\
\rho(\mu_{\alpha,\beta}) & = \frac{\pi^2}{6}(\alpha+\beta) \\
\mathbb{V}_{\mu_{\alpha,\beta}}(\ln x) & = (\alpha-\beta)^2 + \frac{\pi^2}{3}(\alpha\beta+\alpha+\beta)
\end{align}
\end{proposition}

\begin{proof}
These formulas follow easily from Propositions \ref{WLLN_proposition_formula_ln} and \ref{WLLN_proposition_formula_ln2} and Lemma \ref{WLLN_lem_ln_identities}.
\end{proof}

Furthermore, we also can calculate explicitely all fractional moments of $\mu_{\alpha,\beta}$ by the following theorem.

\begin{theorem}
\label{WLLN_thm_moments_mu_ab}
Let $\alpha, \beta > 0$ and $\gamma \in \mathbb{R}$ then we have
\begin{align}
\label{WLLN_moments_mu_ab1}
\int_0^\infty x^\gamma \D \mu_{\alpha,\beta}(x) & = \left\{ \begin{array}{ll} \frac{\sin(\pi\gamma)}{\pi\gamma}\frac{\Gamma(1+\gamma+\gamma\alpha)\Gamma(1-\gamma-\gamma\beta)}{\Gamma(2+\gamma\alpha-\gamma\beta)} & \quad -\frac{1}{1+\alpha} < \gamma < \frac{1}{1+\beta} \\ \infty & \quad \text{otherwise} \end{array} \right. \\
\label{WLLN_moments_mu_ab2}
\int_0^\infty x^\gamma \D \mu_{\alpha,0}(x) & = \left\{ \begin{array}{ll} \frac{\Gamma(1+\gamma+\gamma\alpha)}{\Gamma(1+\gamma)\Gamma(2+\gamma\alpha)} & \quad \gamma > -\frac{1}{1+\alpha} \\ \infty & \quad \text{otherwise} \end{array} \right. \\
\label{WLLN_moments_mu_ab3}
\int_0^\infty x^\gamma \D \mu_{0,\beta}(x) & = \left\{ \begin{array}{ll} \frac{\Gamma(1-\gamma-\gamma\beta)}{\Gamma(1-\gamma)\Gamma(2-\gamma\beta)} & \quad \gamma < \frac{1}{1+\beta} \\ \infty & \quad \text{otherwise.} \end{array} \right. \\
\end{align}
\end{theorem}

\begin{proof}
Let first $-1 < \gamma < 1$. Then \eqref{WLLN_moments_mu_ab1}, \eqref{WLLN_moments_mu_ab2}, and \eqref{WLLN_moments_mu_ab3} follow from Lemma \ref{WLLN_lemma_formula_gamma} together with the formula $\Gamma(1+\gamma)\Gamma(1-\gamma) = \frac{\pi\gamma}{\sin(\pi\gamma)}$.
Since $S_{\mu_{\alpha,0}}(z) = \frac{1}{(z+1)^\alpha}$ is analytic in a neighborhood of $0$, $\mu_{\alpha,0}$ has finite moments of all orders. Therefore the functions
\begin{align}
s & \mapsto \int_0^\infty x^s \D \mu_{\alpha,0}(x) \\
s & \mapsto \frac{\Gamma(1+s+s\alpha)}{\Gamma(1+s)\Gamma(2+s\alpha)}
\end{align}
are both analytic in the halfplane $\Re s > 0$ and they coincide for $s \in (0,1)$. Hence they are equal for all $s \in \mathbb{C}$ with $\Re s > 0$ which proves \eqref{WLLN_moments_mu_ab2}. By Lemma \ref{WLLN_lem_ab_ba} \eqref{WLLN_moments_mu_ab3} follows from \eqref{WLLN_moments_mu_ab2}.
\qed
\end{proof}

\begin{remark}
\label{WLLN_remark_support}
By Theorem \ref{WLLN_thm_moments_mu_ab} \eqref{WLLN_moments_mu_ab1} we have
\begin{enumerate}
\item If $\beta > 0$, then $\int_0^\infty x \D \mu_{\alpha,\beta} = \infty$. Hence $\sup(\text{supp}(\mu_{\alpha,\beta})) = \infty$. Similarly if $\alpha > 0$ then $\int_0^\infty x^{-1} \D \mu_{\alpha,\beta}(x) = \infty$. Hence $\inf(\text{supp}(\mu_{\alpha,\beta}))=0$.
\item If $\beta=0$, then by Stirling's formula 
\begin{align}
\sup(\text{supp}(\mu_{\alpha,0})) = \lim_{0 \to \infty} \left( \int_0^\infty t^n \D \mu_{\alpha,0} \right)^{\frac{1}{n}} = \frac{(\alpha+1)^{\alpha+1}}{\alpha^\alpha}.
\end{align}
Hence by Lemma \ref{WLLN_lem_ab_ba}, we have for $\alpha=0$
\begin{align}
\inf(\text{supp}(\mu_{0,\beta})) = \frac{\beta^\beta}{(\beta+1)^{\beta+1}}.
\end{align}
Note that $\sup(\text{supp}(\mu_{n,0})) = \frac{(n+1)^{n+1}}{n^n}, n \in \mathbb{N}$ was already proven by F. Larsen in \cite[Proposition 4.1]{Larsen2002} and it was proven by  T. Banica, S. T. Belinschi, M. Capitane and B. Collins in \cite{BanicaBelinschiCapitaineCollins2011} that $\text{supp}(\mu_{\alpha,0}) = \left[0,\frac{(\alpha+1)^{\alpha+1}}{\alpha^\alpha}\right]$. Note that this also follows from our Corollary \ref{cor_density_mu_alpha_0}.
\end{enumerate}
\end{remark}

If $\alpha=\beta$ it is also possible to calculate explicitely the density of $\mu_{\alpha,\alpha}$. To do this we require an additional lemma.
\begin{lemma}
\label{WLLN_lem_sin_sin_integral}
For $-1 < \gamma < 1$ and $-\pi < \theta < \pi$ we have
\begin{align}
\frac{\sin \theta}{\pi} \int_0^\infty \frac{t^\gamma}{t^2 + 2 \cos(\theta) t + 1} \D t = \frac{\sin(\theta \gamma)}{\sin(\pi \gamma)}.
\end{align}
\end{lemma}
\begin{proof}
Note first that by the substitution $t = e^x$ we have
\begin{align}
\int_0^\infty \frac{t^\gamma}{t^2 + 2 \cos(\theta) t + 1} \D t = \frac{1}{2} \int_{-\infty}^\infty \frac{e^{\gamma x}}{\cosh x + \cos \theta} \D x.
\end{align}
The function 
\begin{align}
z \mapsto \frac{e^{\gamma x}}{\cosh x + \cos \theta}
\end{align}
is meromorphic with simple poles in $x = \pm \I (\pi - \theta) + p 2 \pi$, $p \in \mathbb{Z}$. Apply now the residue integral formula to this function on the boundary of
\begin{align}
\{z \in \mathbb{C}: -R \le \Re z \le R, 0 \le \Im z \le 2\pi \}
\end{align}
and let $R \to \infty$. The result follows.
\qed
\end{proof}

The density of $\mu_{\alpha,\alpha}$ was computed by P. Biane \cite[Section 5.4]{Biane1998}. For completeness we include a different proof based on Theorem \ref{WLLN_thm_moments_mu_ab} and Lemma \ref{WLLN_lem_sin_sin_integral}.

\begin{theorem}\cite{Biane1998}
\label{WLLN_density_alpha_eq_beta}
Let $\alpha > 0$ then $\mu_{\alpha,\alpha}$ has the density $f_{\alpha,\alpha}(t) \D t$, where
\begin{align}
f_{\alpha,\alpha}(t) = \frac{\sin\left(\frac{\pi}{\alpha+1}\right)}{\pi t \left(t^{\frac{1}{\alpha+1}} + 2 \cos\left(\frac{\pi}{\alpha+1}\right) + t^{-\frac{1}{\alpha+1}} \right)}
\end{align}
for $t \in (0,\infty)$. In  particular $\mu_{1,1}$ has the density $(\pi \sqrt{t}(1+t))^{-1}\D t$ and $\mu_{2,2}$ has the density
\begin{align}
\frac{\sqrt{3}}{2\pi (1 + t^{\frac{2}{3}} + t^{\frac{4}{3}})} \D t.
\end{align}
\end{theorem}

\begin{proof}
To prove this note that for $|\gamma| < \frac{1}{1+\alpha}$ 
\begin{align}
\int_0^\infty x^\gamma f_{\alpha,\alpha}(x) \D x & = \int_0^\infty \frac{\sin\left(\frac{\pi}{\alpha+1}\right)(\alpha+1)y^{\gamma(\alpha+1)}}{\pi\left(y+2\cos\left(\frac{\pi}{\alpha+1}\right) + y^{-1}\right)} \frac{\D y}{y} \\
& = \frac{(\alpha+1)\sin\left(\frac{\pi}{\alpha+1}\right)}{\pi} \int_0^\infty \frac{y^{\gamma(\alpha+1)}}{y^2+2\cos\left(\frac{\pi}{\alpha+1}\right) y + 1} \D y
\end{align}
using the substitution $y = x^{\frac{1}{\alpha+1}}$. Now by Lemma \ref{WLLN_lem_sin_sin_integral} and Theorem \ref{WLLN_thm_moments_mu_ab} \eqref{WLLN_moments_mu_ab1} we have 
\begin{align}
\int_0^\infty x^\gamma f_{\alpha,\alpha}(x) \D x = \int_0^\infty x^\gamma \D \mu_{\alpha,\alpha}(x) < \infty.
\end{align}
This implies by unique analytic continuation that the same formula holds for all $\gamma \in \mathbb{C}$ with $|\Re \gamma| < \frac{1}{\alpha+1}$. In particular
\begin{align}
\int_0^\infty x^{\I s} f_{\alpha,\alpha}(x) \D x = \int_0^\infty x^{\I s} \D \mu_{\alpha,\alpha}(x)
\end{align}
for all $s \in \mathbb{R}$, which shows that the image measures under $x \mapsto \ln x$ of $f_{\alpha,\alpha}(x) \D x$ and $\mu_{\alpha,\alpha}$ have the same characteristic function. Hence $\mu_{\alpha,\alpha} = f_{\alpha,\alpha}(x) \D x$.
\qed
\end{proof}

\begin{proposition}
\label{WLLN_prop_mu_alpha_beta_continious_density}
For all $\alpha,\beta \ge 0$, $(\alpha,\beta)\neq(0,0)$, the measure $\mu_{\alpha,\beta}$ has a continious density $f_{\alpha,\beta}(x)$, $(x>0)$, with respect to the Lebesgue measure on $\mathbb{R}$ and
\begin{align}
\label{WLLN_eq_lim_f_alpha_beta}
\lim_{x \to 0^+} x f_{\alpha,\beta}(x) = \lim_{x \to \infty} x f_{\alpha,\beta}(x) = 0.
\end{align}
\end{proposition}

\begin{proof}
By the method of proof of Theorem \ref{WLLN_density_alpha_eq_beta}, the integral
\begin{align}
h_{\alpha,\beta}(s) = \int_0^\infty x^{\I s} \D \mu_{\alpha,\beta}(x), \quad s \in \mathbb{R}
\end{align}
can be obtained by replacing $\gamma$ by $\I s$ in the formulas \eqref{WLLN_moments_mu_ab1}, \eqref{WLLN_moments_mu_ab2}, and \eqref{WLLN_moments_mu_ab3}. Moreover,
\begin{align}
h_{\alpha,\beta}(s) = \int_0^\infty \exp(\I s t) \D \sigma_{\alpha,\beta}(t)
\end{align}
where $\sigma_{\alpha,\beta}$ is the image measure of $\mu_{\alpha,\beta}$ by the map $x \mapsto \log x$, $(x>0)$. Hence by standard Fourier analysis, we know that if $h_{\alpha,\beta} \in L^1(\mathbb{R})$ then $\sigma_{\alpha,\beta}$ has a density $g_{\alpha,\beta} \in C_0(\mathbb{R})$ with respect to the Lebesgue measure on $\mathbb{R}$ and hence $\mu_{\alpha,\beta}$ has density $f_{\alpha,\beta}(x) = \frac{1}{x} g_{\alpha,\beta}(\log x)$ for $x > 0$, which satisfies the condition \eqref{WLLN_eq_lim_f_alpha_beta}. To prove that $h_{\alpha,\beta} \in L^1(\mathbb{R})$ for all $\alpha,\beta \ge 0$, $(\alpha,\beta)\neq(0,0)$, we observe first that
\begin{align}
\Gamma(1-z)\Gamma(1+z) = \frac{\pi z}{\sin \pi z}, \quad z \in \mathbb{C} \setminus \mathbb{Z}
\end{align}
and hence by the functional equation of $\Gamma$
\begin{align}
\Gamma(2-z)\Gamma(2+z) = \frac{\pi z(1-z^2)}{\sin \pi z}, \quad z \in \mathbb{C} \setminus \mathbb{Z}.
\end{align}

In particular, we have
\begin{align}
|\Gamma(1+\I s)|^2 & = \frac{\pi s}{\sinh \pi s}, \quad s \in \mathbb{R} \\
|\Gamma(2+\I s)|^2 & = \frac{\pi s (1+s^2)}{\sinh \pi s}, \quad s \in \mathbb{R}.
\end{align}

Applying these formulas to \eqref{WLLN_moments_mu_ab1}, \eqref{WLLN_moments_mu_ab2}, and \eqref{WLLN_moments_mu_ab3} with $\gamma$ replaced by $\I s$, we get
\begin{align}
h_{\alpha,\beta}(s) = O\left(|s|^{-3/2}\right), \quad \text{for } s \to \pm \infty 
\end{align}
for all choices of $\alpha,\beta \ge 0$, $(\alpha,\beta)\neq(0,0)$. Thus by the continuity of $h_{\alpha,\beta}$ it follows that $h_{\alpha,\beta} \in L^1(\mathbb{R})$, which proves the proposition.
\qed
\end{proof}

Note that by Remark \ref{WLLN_remark_support} it follows that $f_{\alpha,0}(x)$ can only be non-zero if $x \in \left(0,\frac{(\alpha+1)^{\alpha+1}}{\alpha^\alpha}\right)$ and $f_{0,\beta}(x)$ can only be non-zero if $x \in \left(\frac{\beta^\beta}{(\beta+1)^{\beta+1}},\infty\right)$. Since we have seen, that $\mu_{0,\beta}$ coincides with the stable distribution $v_{\alpha,\rho}$ with $\alpha= \frac{1}{\beta+1}$ and $\rho=1$ we have from \cite[Appendix 4]{BercoviciPataBiane1999} that

\begin{theorem}\cite{BercoviciPataBiane1999}
\label{WLLN_thm_density_0_beta}
The map
\begin{align}
\phi \mapsto \frac{\sin \phi \sin^\beta (\beta \phi)}{\sin^{\beta+1}((\beta+1)\phi)}, \quad 0 < \phi < \frac{\pi}{\beta+1}
\end{align}
is a bijection of the interval $\left(0,\frac{\pi}{\beta+1}\right)$ onto $\left(\frac{\beta^\beta}{(\beta+1)^{\beta+1}},\infty\right)$ and
\begin{align}
\label{WLLN_f_0_beta_density}
f_{\mu_{0,\beta}}\left(\frac{\sin\phi \sin^\beta(\beta \phi)}{\sin^{\beta+1}((\beta+1)\phi)}\right) & = \frac{\sin^{\beta+2}((\beta+1)\phi)}{\pi \sin^{\beta+1}(\beta \phi)}, \quad 0 < \phi < \frac{\pi}{\beta+1}.
\end{align}
\end{theorem}

\begin{proof}
We know that $\mu_{0,\beta} = v_{\frac{1}{\beta+1},1}$, the stable distribution with parameters $\alpha= \frac{1}{\beta+1}$ and $\rho=1$. Moreover, we have from \cite[Proposition A1.4]{BercoviciPataBiane1999}, that $v_{\alpha,1}$ has density $\psi_{\alpha,1}$ on the interval $\left(\alpha(1-\alpha)^{1/\alpha-1},\infty\right)$ given by
\begin{align}
\psi_{\alpha,1}(x) = \frac{1}{\pi} \sin^{1+\frac{1}{\alpha}} \theta \sin^{-\frac{1}{\alpha}} ((1-\alpha)\theta),
\end{align}
where $\theta \in (0,\pi)$ is the only solution to the equation
\begin{align}
x = \sin^{-\frac{1}{\alpha}} \theta \sin^{\frac{1}{\alpha}-1} ((1-\alpha)\theta) \sin \alpha \theta.
\end{align}
It is now easy to check that $f_{0,\beta}(x) = \psi_{\frac{1}{\beta+1},1}(x)$ has the form \eqref{WLLN_f_0_beta_density} by using the substitution $\phi = \frac{\theta}{\beta+1}$.
\qed
\end{proof}

\begin{corollary}
\label{cor_density_mu_alpha_0}
The map
\begin{align}
\phi \mapsto \frac{\sin^{\alpha+1} ((\alpha+1) \phi)}{\sin \phi \sin^{\alpha}(\alpha \phi)}, \quad 0 < \phi < \frac{\pi}{\alpha+1}
\end{align}
is a bijection of the interval $\left(0,\frac{\pi}{\alpha+1}\right)$ onto $\left(0,\frac{(\alpha+1)^{\alpha+1}}{\alpha^\alpha}\right)$ and
\begin{align}
\label{WLLN_f_alpha_0_density}
f_{\mu_{\alpha,0}}\left(\frac{\sin^{\alpha+1} ((\alpha+1) \phi)}{\sin \phi \sin^{\alpha}(\alpha \phi)}\right) & = \frac{\sin^2 \phi \sin^{\alpha-1}(\alpha\phi)}{\pi \sin^{\alpha}((\alpha+1)\phi)}, \quad 0 < \phi < \frac{\pi}{\alpha+1}.
\end{align}
\end{corollary}

\begin{proof}
Since $\mu_{\alpha,0}$ is the image measure of $\mu_{0,\alpha}$ by the map $t \mapsto \frac{1}{t}$, $(t>0)$, we have
\begin{align}
f_{\alpha,0}(x) = \frac{1}{x^2}f_{0,\alpha}\left(\frac{1}{x}\right), \quad x>0.
\end{align}
The corollary now follows from Theorem \ref{WLLN_thm_density_0_beta} by elementary calculations.
\qed
\end{proof}

We next use Biane's method to compute the density $f_{\alpha,\beta}$ for all $\alpha,\beta > 0$.

\begin{theorem}
\label{WLLN_thm_density_mu_alpha_beta}
Let $\alpha,\beta > 0$. Then for each $x > 0$ there are unique real numbers $\phi_1, \phi_2 > 0$ for which
\begin{align}
\pi & = (\alpha+1)\phi_1 + (\beta+1)\phi_2 \label{WLLN_eq_thm_f_alpha_beta_1} \\
x & = \frac{\sin^{\alpha+1} \phi_2}{\sin^{\beta+1} \phi_1}\sin^{\beta-\alpha}(\phi_1+\phi_2). \label{WLLN_eq_thm_f_alpha_beta_2}
\end{align}
Moreover
\begin{align}
\label{WLLN_eq_thm_f_alpha_beta}
f_{\mu_{\alpha,\beta}}\left( x \right) = \frac{\sin^{\beta+2}\phi_1}{\pi \sin^{\alpha}\phi_2} \sin^{\alpha-\beta-1}(\phi_1+\phi_2).
\end{align}
\end{theorem}

\begin{proof}
As $\mu_{\alpha,\beta}$ has the $S$-transform $S_{\mu_{\alpha,\beta}}(z)=\frac{(-z)^\beta}{(1+z)^{\alpha}}$ we by Definition \ref{WLLN_def_S_transform} observe that 
\begin{align}
\chi_{\mu_{\alpha,\beta}}(z) = \frac{-(-z)^{\beta+1}}{(1+z)^{\alpha+1}} \quad \text{whence} \quad \psi_{\mu_{\alpha,\beta}}\left(-\frac{(-z)^{\beta+1}}{(1+z)^{\alpha+1}}\right)=z
\end{align}
for $z$ in some complex neighborhood of $(-1,0)$. Now it is known that
\begin{align}
G_{\mu}\left(\frac{1}{t}\right)=t\left(1+\psi_{\mu}(t)\right)
\end{align}
for every probability measure on $(0,\infty)$. Hence
\begin{align}
\label{WLLN_eq_G_mu_alpha_beta}
G_{\mu_{\alpha,\beta}}\left( - \frac{(1+z)^{\alpha+1}}{(-z)^{\beta+1}} \right)= - \frac{(-z)^{\beta+1}}{(1+z)^\alpha}
\end{align}
for $z$ in a complex neighborhood of $(-1,0)$.

Let $H$ denote the upper half plane in $\mathbb{C}$:
\begin{align}
H = \{ z \in \mathbb{C} : \Im z > 0 \}.
\end{align}

For $z \in H$, put
\begin{align}
\phi_1 &= \phi_1(z) = \arg(1+z) \in (0,\pi) \\
\phi_2 &= \phi_2(z) = \pi - \arg(z) \in (0,\pi).
\end{align}
Basic trigonometry applied to the triangle with vertices $-1$, $0$ and $z$, shows that $\phi_1 + \phi_2 < \pi$ and
\begin{align}
\frac{\sin \phi_1}{|z|} = \frac{\sin \phi_2}{|1+z|} = \frac{\sin(\pi-\phi_1-\phi_2)}{1}.
\end{align}
Hence
\begin{align}
|z| = \frac{\sin \phi_1}{\sin(\phi_1+\phi_2)} \quad \text{and} \quad |1+z| = \frac{\sin \phi_2}{\sin(\phi_1+\phi_2)}
\end{align}
from which
\begin{align}
z = - \frac{\sin \phi_1}{\sin(\phi_1+\phi_2)} e^{\I \phi_2} \quad \text{and} \quad \Im z = \frac{\sin \phi_1 \sin \phi_2}{\sin(\phi_1+\phi_2)}.
\end{align}
It follows that $\Phi \colon z \mapsto (\phi_1(z),\phi_2(z))$ is a diffeomorphism of $H$ onto the triangle $T=\{(\phi_1,\phi_2) \in \mathbb{R}^2 : \phi_1, \phi_2 > 0, \phi_1 + \phi_2 < \pi \}$ with invers
\begin{align}
\Phi^{-1}(\phi_1,\phi_2) = - \frac{\sin \phi_1}{\sin(\phi_1+\phi_2)}e^{-\I \phi_2}, \quad (\phi_1,\phi_2) \in T.
\end{align}

Put $H_{\alpha,\beta} = \{ z \in H : (\alpha+1)\phi_1(z) + (\beta+1)\phi_2(z) < \pi \}$.
Then $H_{\alpha,\beta} = \Phi^{-1}\left( T_{\alpha,\beta} \right)$ where
$T_{\alpha,\beta} = \{(\phi_1,\phi_2) \in T : (\alpha+1)\phi_1 + (\beta+1)\phi_2 < \pi \}.$

In particular $H_{\alpha,\beta}$ is an open connected subset of $H$. Put
\begin{align}
F(z) = - \frac{(1+z)^{\alpha+1}}{(-z)^{\beta+1}}, \quad \Im z > 0.
\end{align}
Then
\begin{align}
\label{WLLN_eq_Fz_alpha_beta}
F(z) = \frac{|1+z|^{\alpha+1}}{|z|^{\beta+1}} e^{\I ((\alpha+1)\phi_1(z) + (\beta+1)\phi_2(z) - \pi)}
\end{align}
so for $z \in H_{\alpha,\beta}$, $\Im F(z) < 0$. Therefore $G_{\mu_{\alpha,\beta}}(F(z))$ is a well-defined analytic function on $H_{\alpha,\beta}$, and since $(-1,0)$ is contained in the closure of $H_{\alpha,\beta}$ it follows from \eqref{WLLN_eq_G_mu_alpha_beta}
\begin{align}
\label{WLLN_eq_G_alpha_beta_F}
G_{\mu_{\alpha,\beta}}(F(z)) = \frac{1+z}{F(z)}
\end{align}
for $z$ in some open subset of $H_{\alpha,\beta}$ and thus by analyticity it holds for all $z \in H_{\alpha,\beta}$.

Let $x > 0$ and assume that $\phi_1, \phi_2 > 0$ satisfy \eqref{WLLN_eq_thm_f_alpha_beta_1} and \eqref{WLLN_eq_thm_f_alpha_beta_2}. Put
\begin{align}
z = \Phi^{-1}(\phi_1,\phi_2) = - \frac{\sin \phi_1}{\sin(\phi_1+\phi_2)} e^{-\I\phi_2}.
\end{align}
Then by \eqref{WLLN_eq_Fz_alpha_beta}
\begin{align}
F(z) = \frac{|1+z|^{\alpha+1}}{|z|^{\beta+1}} = \left( \frac{\sin \phi_2}{\sin(\phi_1+\phi_2)}\right)^{\alpha+1} \left( \frac{\sin(\phi_1+\phi_2)}{\sin \phi_1} \right)^{\beta+1} = x.
\end{align}
Since $\mu_{\alpha,\beta}$ has a continious density $f_{\alpha,\beta}$ on $(0,\infty)$ by Proposition \ref{WLLN_prop_mu_alpha_beta_continious_density}, the inverse Stieltjes transform gives
\begin{align}
f_{\alpha,\beta}(x) = -\frac{1}{\pi} \lim_{w \to x, \Im w > 0} \Im G_{\mu_{\alpha,\beta}}(w) = \frac{1}{\pi} \lim_{w \to x, \Im w < 0} \Im G_{\mu_{\alpha,\beta}}(w).
\end{align}
For $0<t<1$, put $z_t = \Phi^{-1}(t\phi_1,t\phi_2)$. Then
\begin{align}
z_t \in \Phi^{-1}\left(T_{\alpha,\beta}\right) = H_{\alpha,\beta}.
\end{align}
Thus $\Im F(z_t) < 0$. Moreover, $ z_t \to z$ and $F(z_t) \to F(z) = x$ for $t \to 1^-$. Hence by \eqref{WLLN_eq_G_alpha_beta_F},
\begin{align}
f_{\alpha,\beta}(x) = \frac{1}{\pi} \lim_{t \to 1^-} \Im G_{\mu_{\alpha,\beta}}(F(z_t)) = \frac{1}{\pi} \lim_{t \to 1^-} \Im \left( \frac{z_t+1}{F(z_t)} \right) = \frac{\Im z}{\pi x} = \frac{\sin \phi_1 \sin \phi_2}{\pi x \sin(\phi_1+\phi_2)}
\end{align}
which proves \eqref{WLLN_eq_thm_f_alpha_beta}. To complete the proof of Theorem \ref{WLLN_thm_density_mu_alpha_beta}, we only need to prove the existence and uniqueness of $\phi_1, \phi_2 > 0$. Assume that $\phi_1, \phi_2$ satisfy \eqref{WLLN_eq_thm_f_alpha_beta_1} then
\begin{align}
\phi_1 = \frac{\pi - \theta}{\alpha+1} \quad \text{and} \quad \phi_2=\frac{\theta}{\beta+1}
\end{align}
for an unique $\theta \in (0,\pi)$. Moreover,
\begin{align}
\frac{\D \phi_1}{\D \theta} = -\frac{1}{\alpha+1} \quad \text{and} \quad \frac{\D \phi_2}{\D \theta} = \frac{1}{\beta+1}.
\end{align}
Hence, expressing $u = \frac{\sin^{\alpha+1} \phi_2}{\sin^{\beta+1} \phi_1} \sin^{\beta-\alpha}(\phi_1+\phi_2)$ as a function $u(\theta)$ of $\theta$, we get 
\begin{align}
(\alpha+1)(\beta+1) \frac{\D u(\theta)}{\D \theta} & = (\beta+1)^2 \cot \phi_1 + (\alpha+1)^2 \cot \phi_2 - 2(\alpha-\beta)^2 \cot(\phi_1+\phi_2) \\
& = \frac{A(\phi_1,\phi_2)}{\sin \phi_1 \sin \phi_2 \sin(\phi_1+ \phi_2)}
\end{align}
where
\begin{align}
A(\phi_1,\phi_2) = \left((\alpha+1)\sin\phi_1\cos\phi_2+(\beta+1)\cos\phi_1\sin\phi_2\right)^2+(\alpha-\beta)^2 \sin^2\phi_1 \sin^2\phi_2.
\end{align}
For $\alpha \neq \beta$ $A(\phi_1,\phi_2) \ge (\alpha-\beta)^2 \sin^2 \phi_1 \sin^2 \phi_2 > 0$ and for $\alpha=\beta$ $A(\phi_1,\phi_2) = (\alpha+1)^2 \sin(\phi_1+\phi_2) > 0$. Hence $u(\theta)$ is a differentiable, strictly increasing function of $\theta$, and it is easy to check that
\begin{align}
\lim_{\theta \to 0^+} u(\theta) = 0 \quad \text{and} \quad \lim_{\theta \to \pi^-} u(\theta) = \infty.
\end{align}
Hence $u(\theta)$ is a bijection of $(0,\pi)$ onto $(0,\infty)$, which completes the proof of Theorem \ref{WLLN_thm_density_mu_alpha_beta}.
\qed
\end{proof}

\begin{remark}
It is much more complicated to express the densities $f_{\alpha,\beta}(x)$ directly as functions of $x$. This has been done for $\beta=0$, $\alpha \in \mathbb{N}$ by K. Penson and K. \.{Z}yczkowski in \cite{PensonZyczkowski2011} and extended to the case $\alpha \in \mathbb{Q}^+$ by W. M{\l}otkowski, K. Penson and K. \.{Z}yczkowski in \cite[Theorem 3.1]{MlotkowskiPensonZyczkowski2012p}.
\end{remark}

\bibliography{WLLN}

\end{document}